%% file: Tits_rigid_201403.tex
\documentclass[amscd,amssymb,verbatim,12pt]{amsart}
\usepackage{pdfsync}
\newif\ifAMS
\IfFileExists{amssymb.sty}
  {\AMStrue\usepackage{amssymb}}
  {\usepackage{latexsym}}
\theoremstyle{plain}

\newtheorem{Thm}{Theorem}
\newtheorem{Cor}[Thm]{Corollary}
\newtheorem{Lem}[Thm]{Lemma}
\newtheorem{Prop}[Thm]{Proposition}

\theoremstyle{definition}
\newtheorem*{Def}{Definition}

\theoremstyle{remark}

\newcommand{\interior}{^{ \kern-5pt ^\circ}}
\newcommand {\bd}{\partial}

\newcommand {\iy}{\infty}
\newcommand {\cl}{\overline}

\newcommand {\R}{{\mathbb R}}
\newcommand {\Z}{{\mathbb Z}}

\newcommand {\E}{{\mathbb E}}
\newcommand {\HH}{{\mathbb H}}

\newcommand {\diam}{\text{diam}}

\newcommand {\fp}{\text {Fix}}

\newcommand {\cD}{ {\mathcal  D}}
\newcommand {\cC}{{\mathcal  C}}

\newif\ifPDF
\ifx\pdfoutput\undefined
    \PDFfalse
\else
    \ifnum\pdfoutput > 0
        \PDFtrue
    \else
        \PDFfalse
    \fi
\fi \ifPDF
    \usepackage[pdftex]{graphicx}
    \DeclareGraphicsExtensions{.pdf,.png,.jpg}
    \graphicspath{{}}
\else
    \usepackage{graphicx}
    \DeclareGraphicsExtensions{.eps}
    \graphicspath{{}}
\fi

\usepackage{color}

\begin{document}
\title{Tits rigidity of CAT(0) group boundaries}

\author
{Khek Lun Harold Chao and Eric Swenson }


\email {khchao@indiana.edu} \email {eric@math.byu.edu}

\address
[Khek Lun Harold Chao] {Department of Mathematics, Indiana University, Bloomington, IN 47405 }
\address
[Eric Swenson] {Mathematics Department, Brigham Young University,
Provo UT 84602}
\thanks {This work was partially supported by a grant from the Simons Foundation (209403)}

\begin{abstract}
We define Tits rigidity for visual boundaries of CAT(0) groups, and prove that the join of two Cantor sets and its suspension are Tits rigid.
\end{abstract}

\maketitle
\section {Introduction}

A CAT(0) space  $X$ has two natural boundaries with the same underlying point set, the visual boundary, $\bd X$ and the Tits boundary, $\bd_T X$ . The obvious bijection from $\bd_T X$ to $\bd X$  is continuous, but need not be a homeomorphism.

In the classical case where $X$ is a Riemannian $n+1$-manifold of non-positive sectional curvature, then $\bd X = S^n$, so the visual boundary contains very little information (only the dimension).  The Tits boundary, on the other hand, is much more interesting.  For example the Tits boundary of $\E^{n+1}$ is also $S^n$, while the Tits boundary of $\HH^{n+1}$ is discrete.  These are of course different for $n >0$.  Even in the case where $n=2$, there are at least two other possible Tits boundaries: The Tits boundary of $\HH^2 \times \R$ is the spherical suspension of an uncountable discrete set; and the examples of Croke and Kleiner give the infamous eye of Sauran pattern.  In this paper we examine the other extreme, where the visual topology dictates the Tits metric.

Suppose that $X$ admits a geometric group action. Ruane showed in \cite{RUA1} that if $\bd X$ is a suspension of Cantor set, then $\bd_T X$ is the spherical suspension of an uncountable discrete set. In \cite{CHA} it is shown that if $\bd X$ is the join of two Cantor sets, and if $X$ admits a geometric action by a group $G$ that contains $\Z^2$, then $\bd_T X$ is isometric to the spherical join of two uncountable discrete sets.

In this paper, we prove that the same result holds without the $\Z^2$ assumption on $G$. We use the action of ultrafilters over $G$ on $\bd X$, whose  properties were investigated in the paper \cite{GU-SW}. We will also show that if $\bd X$ is the suspension of a join of two Cantor sets, then $\bd_T X$ is also the spherical suspension of the  spherical join of two uncountable discrete sets. These results suggest the definition of Tits rigidity, and the above results can be rephrased into saying that the suspension of a Cantor set, the join of two Cantor sets and the suspension of the  join of two Cantor sets are Tits rigid. On the other hand, a sphere of dimension $n>0$  is not Tits rigid since, as we saw,  $\E^{n+1}$ and $\HH^{n+1}$ have different Tits boundaries.

The organization of this paper is as follows: This section reviews some basic notions and defines Tits rigidity; section 2 completes the proof  that the join of two Cantor set is Tits rigid; and section 3 proves that the suspension of the join of two Cantor set is Tits rigid. We states some questions in section 4.

We refer the reader to \cite{BRI-HAE} or \cite{BAL} for more details.
\begin{Def} For  $X$ a metric space, and $I$ interval of $\R$, an isometric embedding $\alpha:I \to \R$ is called a geodesic.  By abuse of notation we will also refer to the image of $\alpha$ as a geodesic.
\end{Def}

\begin{Def} For $X$ a geodesic metric space and $\Delta(a,b,c)$ a geodesic triangle in $X$ with vertices $a,b,c \in X$ there is an Euclidean  {\em comparison }triangle $\bar \Delta=\Delta(\bar a,\bar b, \bar c) \subset \E^2$ with
$d(a,b) =d(\bar a, \bar b)$, $d(a,c) = d(\bar a, \bar c)$ and $d(b,c)=d(\bar b, \bar c)$.
We define the comparison angle $\bar \angle_a(b,c) =\angle_{\bar a}(\bar b,\bar c)$.

Each point $z \in \Delta(a,b,c)$ has a unique comparison point, $\bar z \in \bar \Delta$.  We say that the triangle $\Delta(a,b,c)$ is CAT(0) if for any $y, z \in \Delta(a,b,c)$ with comparison points $\bar y, \bar z \in \bar \Delta$, $d(y,z) \le d(\bar y,\bar z)$. The space $X$ is said to be CAT(0) if every geodesic triangle in $X$ is CAT(0).
\end{Def}
If $X$ is CAT(0), notice that for any geodesics $\alpha:[0, r] \to X$ and $\beta:[0,s] \to X$ with $\alpha(0)=\beta(0)=a$, the function $$\theta(r,s) =\bar \angle_{a}(\alpha(r),\beta(s))$$ is an increasing function of  $r,s$.  Thus $\lim\limits_{r,s \to 0} \theta(r,s)$ exists and we call this limit $\angle_a(\alpha(r),\beta(s))$.  It follows that for any $a,b,c \in X$, a CAT(0) space, $$\angle_a(b,c) \le \bar \angle_a(b,c).$$

Recall that a metric space is {\em proper} if closed metric balls are compact.  Recall that an action by isometries of a group $G$ on a space $X$ is {\em geometric} if the action is properly discontinuous and cocompact.

{\bf For the duration, $G$ will be  a group and $X$ a proper CAT(0) space on which $G$ acts geometrically.  }

The
(visual) boundary, $\bd X$, is the set of equivalence classes of
rays, where rays  are equivalent if they are within finite Hausdorff distance from each other.  Given a
ray $R$ and a point $x \in X$, there is a ray $S$ emanating from
$x$ with $R \sim S$.  Fixing a base point $\mathbf 0 \in X$, we
define  the visual topology on $\bar X= X \cup \bd X$ by taking the basic
open sets of $ x \in X$ to be the open metric balls about $x$. For
$y \in \bd X$, and $R$ a ray from $\mathbf 0$ representing $y$, we
construct  basic open sets $U(R,n,\epsilon)$ where $n,\epsilon>0$.
We say $z \in U(R,n,\epsilon)$ if the unit speed geodesic,
$S:[0,d(\mathbf 0,z)] \to \bar X$, from $\mathbf 0$ to $z$
satisfies $d(R(n),S(n)) <\epsilon$.  These sets form a basis for a regular
topology on $\bar X$ and $\bd X$. For any $x \in X$ and $u,v \in \bd X$, we can define
$\angle_x(u,v) $ and $\bar \angle_x(u,v)$ by parameterizing the rays $[x,u)$ and $[x,v)$ by
$t\in [0,\infty)$ and taking the limit of $\bar \angle_x$ as $t\to 0$ and $t \to \iy$ respectively.

 For $u,v \in \bd X$, we define $\angle(u,v) =
\sup\limits_{p \in X} \angle_p(u,v)$.
 It follows from \cite{BRI-HAE} that $\angle(u,v) =
\overline{\angle}_p(u,v)$ for any $p \in X$.
 Notice that isometries of $X$ preserve the angle between points of $\bd X$.
 This defines a metric called the angle metric on the set $\bd X$.
 The angle metric defines a path metric, $d_T$ on the set $\bd X$, called the Tits
metric, whose topology is at least as fine as the visual topology of $\bd
X$.  Also $\angle(a,b)$ and $ d_T(a,b)$ are equal whenever either
of them is less than $\pi$.  For any $u \in \bd X$, we define $B_T(u, \epsilon) = \{v \in \bd X : d_T(u,v) < \epsilon\}$
and $\bar B_T(u, \epsilon) = \{v \in \bd X : d_T(u,v) \le \epsilon\}$.

  The set $\bd X$ with the Tits metric is called the Tits  boundary of $X$,
denoted  $\bd_TX$.  Isometries of $X$ extend to isometries of $\bd_TX$.

  The identity function $\bd_T X \to \bd X$ is continuous, but the identity
function
 $\bd X \to \bd_TX$ is only lower semi-continuous.  That is for any sequences
$(u_n), (v_n)  \subset \bd X$ with $u_n \to u$ and $v_n \to v$ in
$\bd X$, then $$\varliminf d_T(u_n,v_n) \ge d_T(u,v)$$

\begin{Def} A subgroup $H < G$ is called convex if there exists closed convex $A\subset X$ with $H$ acting on $A$ geometrically.
\end{Def}
\begin{Def} For $g\in G$, we define $\tau(g) = \inf\limits_{x \in X} d(x,g(x))$.  This minimum is realized and  Min$(g)= \{x \in X|\, d(x,g(x)) = \tau(g)\}$ is nonempty.
\end{Def}
For any $g \in G$,  the centralizer $Z_g$ is a convex subgroup that acts geometrically on Min$(g)$, which is closed and convex by \cite{RUA},\cite{BRI-HAE}. In fact if $g$ is hyperbolic, then Min$(g) = A \times Y$ where $Y$ is a closed convex subset of $X$ on which $Z_g/\langle g\rangle$ acts geometrically (\cite{BRI-HAE},\cite{SWE}), and $A$ is an axis of $g$.

\begin{Def} The boundary of a CAT(0) space will be called a CAT(0) boundary.  If $G$ is a group acting geometrically on a CAT(0) space $X$, then $\bd X$ is called a CAT(0) boundary of $G$, or we say $\bd X$ is a CAT(0) group boundary. In all cases a CAT(0) boundary comes equipped with both  the visual topology and the Tits metric (which normally gives a finer topology).
\end{Def}
\begin{Def}  Let $A$ and $B$ be boundaries of CAT(0) spaces. A function $f:A \to B$ is called a {\em boundary isomorphism}  if $f$ is a homeomorphism in the visual topology and
$f$ is an isometry in the Tits metric.   A function $g:A \to B$ is called a {\em boundary embedding} if $g$ is a boundary isomorphism onto its image, where the metric on $g(A)$ is the restriction of the Tits metric.
\end{Def}
Two boundaries of the same CAT(0) group need not be boundary isomorphic to each other or even homeomorphic to each other \cite{CRO-KLE}.
\begin{Def} For $A \subset X$, $\Lambda A$ is the set of limit points of $A$ in $\bd X$.
For $H <G$,  $\Lambda H = Hx$ where $Hx$ is the orbit of some $x\in X$ ( this is independent of the choice of $x$).
\end{Def}
\begin{Lem}\label{L:con}Let $Y$ be a closed convex subset of $X$. Inclusion of $Y$ into $X$ induces
 $\iota: \bd Y \to \Lambda Y$, a topological embedding of $\bd Y$ in $\bd X$.  Also $\iota$ is  isometric on the angle metric.
 Furthermore if $\diam\, \bd_T Y \le\pi$, then $\iota:\bd_TY \to \bd _TX$ is a boundary embedding.
 \end{Lem}\label{L:iota}
 \begin{proof} Since the inclusion $Y \to X$ is isometric, geodesics in $Y$ are geodesics in $X$, so by choosing a base point $y \in Y$, we have $\bd Y \subset \bd X$ which defines $\iota$.  Also for $R$ a geodesic in $Y$ from $y$, $U_X(R,n,\epsilon) \cap Y = U_Y(R,n,\epsilon)$, where $U_X$ is the neighborhood in $X$ and $U_Y$ the neighborhood in $Y$.  Thus $\iota: \bd Y \to \bd X$ is an embedding with image $\Lambda Y$.   Since $Y$ is isometrically embedded in $X$, for any $\alpha,\beta \in \Lambda Y$, $\overline{\angle}_y(\alpha ,\beta)$ is the same in both $X$ and $Y$.  It follows that $\iota
 $ is isometric on the angle metric, and  so $\iota: \bd_T Y \to \bd_T X$ will be Lipschitz 1.

 Now suppose $\diam\, \bd_T Y \le\pi$. This implies that the set $S=\{(\alpha,\beta) \in \bd Y \times \bd Y:\, d_T(\alpha,\beta)< \pi\}$ is dense in $\bd_T Y \times \bd_T Y$.
Since the angle metric and the Tits metric are the same when either is less than $\pi$, $ d_T\circ (\iota\times \iota) = d_T$ on $S$ and it follows that $\iota:\bd_TY \to \bd _TX$ is an isometric embedding.
 \end{proof} A line in the Euclidean plane gives an example of when $\iota$ is not a boundary embedding.

\begin{Def}
A compact metrizable  space  $Y$ is said to be \textbf{Tits rigid}, if for any two CAT(0) group boundaries $Z_1$ and $Z_2$ homeomorphic to $Y$, $Z_1$ is boundary isomorphic to $Z_2$.
\end{Def}
\begin{Def} \cite{BRI-HAE} For $Y_1$, $Y_2$ topological spaces we define their topological join $Y_1*Y_2$ to be the quotient of $ Y_1\times Y_2 \times [0,\frac \pi 2] $
modulo
$(a,b,0) \sim (a,c,0)$ and $( a,c,\frac \pi 2) \sim ( b,c,\frac \pi 2)$.  We will refer to $Y_1\times Y_2\times \{0\}$ as $Y_1$ and we will refer to $Y_1\times Y_2\times \{\frac \pi 2\}$ as $Y_2$.
For fixed $y_i\in Y_i$, the arc $(y_1,y_2,t), t\in [0,\frac \pi 2]$ will be called the join arc from $y_1$ to $y_2$.

  For $Y_1$ and $Y_2$ metric spaces with metrics  bounded by $\pi$, the spherical join $Y_1*_SY_2$ is the  point set $Y_1*Y_2$ endowed with the metric
$$\aligned &d((y_1,y_2,\theta), (y_1',y_2',\theta')) =\\ &\arccos\left[\cos \theta \cos \theta' \cos(d(y_1,y_1'))+ \sin \theta \sin \theta' \cos(d(y_2,y_2'))\right]\endaligned$$
For $Y$, a topological space, we define the suspension $\sum Y$ to be the topological join of $Y$ with $\{n,p\}$, a discrete two point set. In this setting we refer to the join arcs as suspension arcs.
 For $Y$ a metric space with metric bounded by $\pi$, we define the spherical suspension $\sum_S Y$ to be the spherical join of $Y$ with $\{n,p\}$ where $d(n,p)$ is defined to be $\pi$.
\end{Def}
\section{The Join of  two Cantor sets is Tits rigid}

Suppose that $\bd X$ is topologically the suspension of two Cantor sets $C_1$ and $C_2$, so $\bd X \cong C_1 * C_2$. Replacing $G$ with a subgroup of index at most 2, {\bf we may assume that $C_1$ and $C_2$ are $G$ invariant. } By \cite{CHA} the action of $G$ on $X$ is not rank 1.
Thus by \cite{PA-SW}, since $\bd X$ is one dimensional, then $\diam(\bd_T X) \le \frac {4\pi}3$. If $g \in G$ with $\{g^\pm\} \not \subset C_i$ for $i=1,2$, then we are done by \cite{CHA}. Thus we may assume that there are infinitely many hyperbolic $g \in G$ with
$\{g^\pm \} \subset C_1$.  Let $\alpha = d_T(C_1,C_2)$.

 By compactness, there will be points of $ C_1$ and $ C_2$ realizing this minimum.   Since $C_1$ and $C_2$ are closed invariant subsets of $\bd X$, then for any $p \in \bd X$ and $i=1,2$, $d(p,C_i) \le \frac \pi 2$ by \cite{PA-SW} so $\alpha \le \frac \pi 2$.
For  $a\in C_i$  and  $b\in C_{3-i}$, let $\overline{ab}$ be the suspension arc from $a$ to $b$.  For any path $\gamma$ in $\bd(X)$
let $\ell(\gamma)$ be the Tits length of the path $\gamma$ (which may be $\iy$).
\begin{Lem}\label{L:pied}
Let $a \in C_i$ for $i=1,2$ and $b \in C_{3-i}$.  There exists $c \in C_{3-i}-\{b\}$ such that $\ell(\overline{ab})+\ell(\overline{ac}) \le \pi$.
\end{Lem}
\begin{proof} Suppose not, then for all $c \in C_{3-i}-\{b\}$, $\ell(\overline{ab})+\ell(\overline{ac})> \pi$.

First consider the case where $\ell(\overline{ab}) >\frac \pi 2$.  By lower semi-continuity, $d_T(a, C_{3-i}) + \ell(\overline{ab}) >\pi$.
We can choose $p\in \overline{ab}$ with $\ell(\overline{pb})> \frac \pi 2$ and $\overline{ap} + d_T(a, C_{3-i}) > \frac \pi 2$.  Thus
$d(p,C_{3-i})> \frac \pi 2$, a contradiction.

Now consider the case where  $\ell(\overline{ab}) \le\frac \pi 2$ (It follows that $\ell(\overline{ab})=d_T(a,C_{3-i})$).  Thus for any $c \in C_{3-i}-\{b\}$ there is a point $p \in \overline{ac}$
with $\ell(\overline{pc})> \frac \pi 2$ and $\ell(\overline{ap})+ \ell(\overline{ab}) > \frac \pi 2$.  It follows that $d_T(p,C_{3-i})> \frac \pi 2$ , a contradiction.
\end{proof}
We then get the following obvious consequence.
\begin{Cor} \label{C:bound} For any $a \in C_1$ and $b \in C_2$, $\ell\left(\overline {ab}\right) \le \pi -\alpha$.
\end{Cor}
\begin{Lem} \label{L:contr} Suppose for some $b\in C_i$ $i=1,2$,  $d_T(b, C_{3-i}) > \alpha$.  Then
\begin{enumerate}
\item $\alpha < \frac \pi 4$
\item $d_T(b,C_{3-i}) \le \frac \pi 2 - \alpha$
\item $\ell(\overline{bc}) \le \pi - 2\alpha -d_T(b, C_{3-i})$ for all $c \in C_{3-i}$.
\end{enumerate}
\end{Lem}
\begin{proof} The subset $A_i =\{a\in C_i |\, d_T(a, C_{3-i})=\alpha\}$ is closed and $G$ invariant.  It follows that $$\frac \pi 2 \ge d_T(b, A_i) \ge d_T(b,C_{3-i}) +d_T(C_{3-i}, A_i)=  d_T(b,C_{3-i})  + \alpha$$  and we have (1) and (2).

Now let $c \in C_{3-i}$.
If $\ell(\overline{bc}) > \pi - 2\alpha -d_T(b, C_{3-i})$, then there is a point $p \in \overline{bc}$ with $\ell(\overline{pc}) +\alpha > \frac \pi 2$ and $ \ell(\overline{bp}) +d_t(b,C_{3-i}) +\alpha > \frac \pi 2$.  It follows that $d_T(p, A_i) > \frac \pi 2$, a contradiction.
\end{proof}
\begin{Def}Let $\beta G$ be the set of all ultrafilters on $G$, and for $\omega \in \beta G$, and $z \in \bd X$, define
$T^\omega(z)=\lim\limits_{g\to \omega} g(z)$. Recall that for each $U$ open set of $\bd X$ with
$T^\omega (z)\in U$, we have $\omega \{g\in G:\, g(z) \in U\}=1$. Thus gives a function $T^\omega:\bd X \to \bd X$ which is not continuous (in general) but is Lipschitz 1 in the Tits metric (see \cite{GU-SW}).
\end{Def}
One might think that $\bd X -(C_1\cup C_2)$ was invariant under $T^\omega$, but this isn't a priori true.  We do however have the following:
\begin{Thm} \label{T:omega}
Let $\omega$ be an ultrafilter on $G$ and $c_i \in C_i$ for $i=1,2$.  Let $\hat c_i = T^\omega(c_i)$ for $i=1,2$.
Then $T^\omega(\cl{c_1c_2}) =\cl{\hat c_1\hat c_2}$.
\end{Thm}
\begin{proof}   Let $\pi: C_1 \times C_2 \times [0,1] \to C_1*C_2$ be the quotient map. Since $C_i$ is $G$ invariant for $i =1,2$, $T^\omega(C_i) \subset C_i$.

Notice that for each $g \in G$, $g(\cl{c_1c_2})=\cl{g(c_1)g(c_2)}$.  Suppose that for some $b \in [c_1,c_2]$, $\hat b =T^\omega(b) \not \in \cl{\hat c_1 \hat c_2}$.  Then there exists  open neighborhood $U_i $ of $ \hat c_i $ in $C_i$, for $i=1,2$  and open $V \ni \hat b$ of $C_1*C_2$  with
$\pi(U_1 \times U_2 \times [0,1]) \cap V = \emptyset$.  Notice that $\omega \{g \in G:\, g(c_i) \subset U_i\}= 1 $ for $i=1,2$.  However
$\{g \in G :\, g(b) \in \pi(U_1 \times U_2 \times [0,1])\} = \bigcap\limits_{i=1}^2\{g \in G:\, g(c_i) \subset U_i\}$ and so $\omega \{g \in G :\, g(b) \in \pi(U_1 \times U_2 \times [0,1])\} =1$.  It follows that $\omega \{g \in G:\, g(b) \in V\} =0$ which is a contradiction. It follows that $T^\omega(\cl{c_1c_2}) \subseteq\cl{\hat c_1\hat c_2}$.  However by Lemma \ref{L:pied},  $\ell(\cl{c_1c_2}) \le \pi$, so $\cl{c_1c_2}$ is connected in the Tits metric, and since $T^\omega$ is Lipschitz on $\bd_TX$,  $T^\omega(\cl{c_1c_2})$ is connected and therefore $T^\omega(\cl{c_1c_2}) =\cl{\hat c_1\hat c_2}$.

\end{proof}
\begin{Lem}\label{L:hyper} Let $g\in G$ hyperbolic with $\{g^\pm \} \subset C_1$.  If $\alpha < \frac \pi 2$, then there are infinitely many
$c\in C_2$ with $\ell(\cl{g^+c}) = d_T(g^+, C_2)$.
\end{Lem}
\begin{proof}
By lower semi-continuity, there exists $c \in C_2$ with $\ell(\cl{g^+c}) = d_T(g^+, C_2)$.
If any positive power of $g$ fixes  $c$, then by \cite{SWE}, there is $\Z^2 <G$ and  by \cite{CHA} $ \alpha = \frac \pi 2$.
Thus the $\langle g \rangle$ orbit of $c$ is infinite, and all points $b$ in this orbit satisfy $\ell(\cl{g^+b}) = d_T(g^+, C_2)$.
\end{proof}

\begin{Thm}\label{T:pi/2} $\alpha = \frac \pi 2$.
\end{Thm}
\begin{proof}
We assume $\alpha < \frac \pi 2$.

Case I: For any $a \in C_1$, $d_T(a, C_2)= \alpha$.   Using Lemma \ref{L:hyper} there exists $b \in C_1$ and distinct $c,d \in C_2$ with
$\ell\left(\cl{b c}\right) =\alpha = \ell\left(\cl{bd}\right)$.  Choose $a \ne b$, with $a \in C_1$ and then choose $e \in C_2$ with $\ell\left(\cl{ae}\right) = \alpha$.
By Corollary \ref{C:bound},  $\ell\left(\cl{ac}\right), \ell\left(\cl{ad}\right), \ell\left(\cl{be}\right) \le \pi - \alpha$.   Each of the  loops $aebd$, $aebc$ and $adbc$  which is non-trivial must have length at least $2\pi$.  It follows that $\ell\left(\cl{ac}\right), \ell\left(\cl{ad}\right), \ell\left(\cl{be}\right) = \pi - \alpha$.
 Let $m$ be the midpoint of the segment $\cl{bc}$.  Let $\omega \in \beta G$ be an ultrafilter pulling from $m$.  Let $ T^\omega(a) =\hat a$, $ T^\omega(b) =\hat b$, $ T^\omega(c) =\hat c$, $ T^\omega(d) =\hat d$, $ T^\omega(e) =\hat e$, and $ T^\omega(m) =\hat m$.  By \cite{GU-SW}, $T^\omega$ is an isometry on each Tits segment  of length at most $\pi$ from $m$ and
 $T^\omega(\bd X)$ is contained in the set of all Tits geodesics of length $\pi$ from $\hat m$ to some point $\hat p \in \bd X$.  Since these geodesics can branch only at $\hat m$ and $\hat p$, it follows $T^\omega\left(\cl{bd}\right) \subset T^\omega\left(\cl{be}\right)$.  However by Theorem \ref{T:omega}, $T^\omega\left(\cl{bd}\right)=\cl{\hat b \hat d}$ and $T^\omega\left(\cl{be}\right)= \cl{\hat b \hat e}$.  Thus $\cl{\hat b \hat d}\subset  \cl{\hat b \hat e}$ and it follows by definition that $\hat d = \hat e$.
\begin{figure}[h]
\centering
\def\svgwidth{0.65\columnwidth}
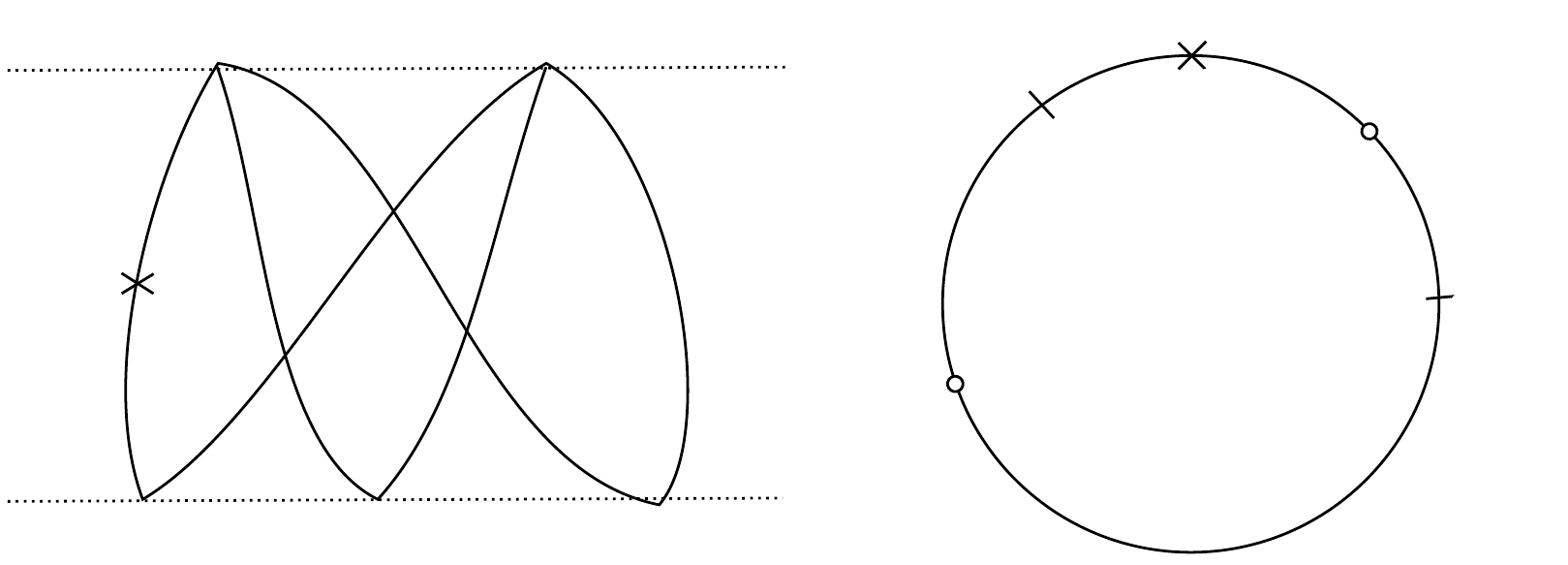
\vspace{.5in}
\def\svgwidth{0.65\columnwidth}
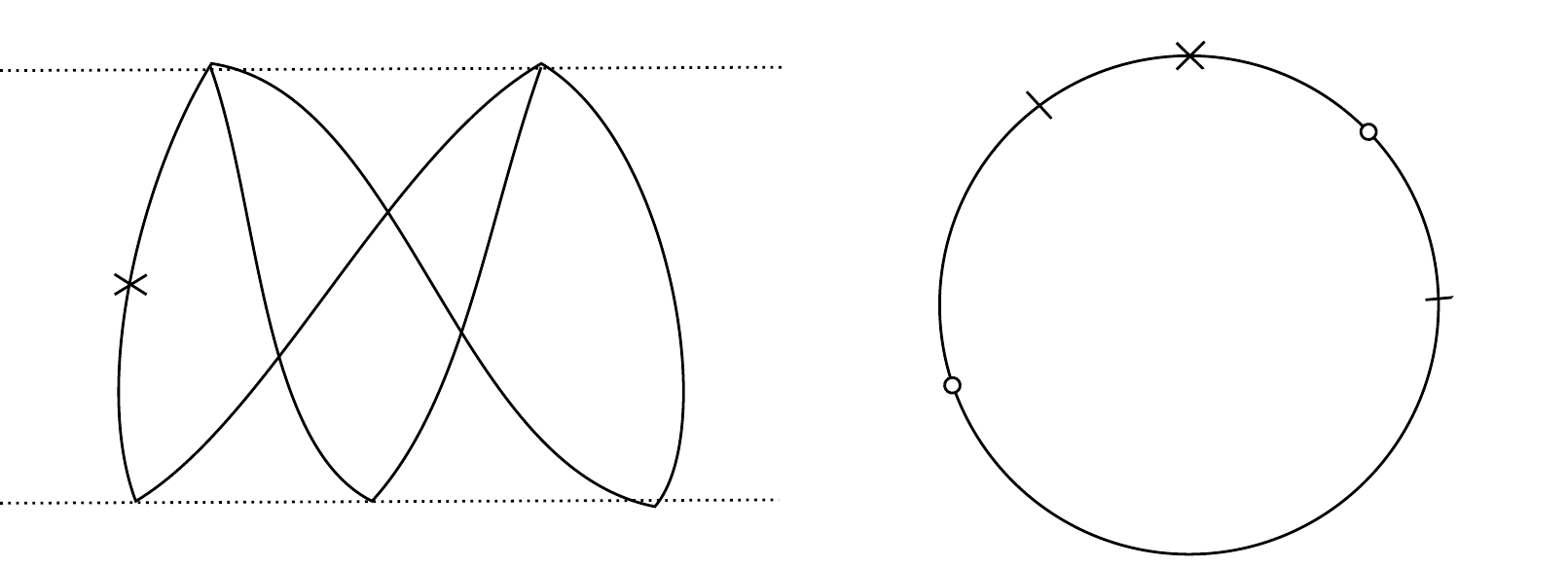
\caption{Proof of Theorem \ref{T:pi/2}}
\label{casespic}
\end{figure}

 Since $T^\omega$ is an isometry on any  Tits segment from $m$ of length at most $\pi$, and $d_T(c,m) < d_T(d,m) \le \pi$, $\hat c \neq \hat d$ and similarly $\hat a\neq \hat b$.  Thus the  loop $\hat a\hat c \hat b \hat d$ is nontrivial.  Since $T^\omega$ is Lipschitz with constant one on the Tits metric, then $\ell\left(\cl{\hat b \hat d}\right) =\alpha =\ell\left(\cl{\hat b \hat c}\right) $.  Since $\hat e =\hat d$, $\ell\left(\cl{\hat a \hat d}\right) =\alpha$ and finally $\ell\left(\cl{\hat a \hat c}\right) \le \pi - \alpha$.  Thus the nontrivial loop  $\hat a\hat c \hat b \hat d$
 has $\ell(\hat a\hat c \hat b \hat d) \le  3\alpha +\pi -\alpha = \pi +2\alpha < 2\pi$ since $\alpha < \frac \pi 2$.  This is a contradiction.

 Case II: There is $a \in C_1$ with $d_T(a,C_2) > \alpha$.  Let $g\in G$ a hyperbolic with  $\{g^\pm \} \subset C_1$, and let $b = g^+$ and $\beta = d_T(b,C_2)$.  Notice by Lemma  \ref{L:contr},  $\beta < \frac \pi 2$. Using Lemma \ref{L:hyper} there are distinct $c,d \in C_2$ with
$\ell\left(\cl{b c}\right) =\beta = \ell\left(\cl{bd}\right)$. Choose $a \ne b$, with $a \in C_1$ and then choose $e \in C_2$ with $\ell\left(\cl{ae}\right) = \alpha$.
We now proceed as in Case I pulling from $m$ the mid point of $\cl{bc}$. We obtain as before a nontrivial loop $\hat a\hat c \hat b \hat d$. However  this time we have $\ell\left(\cl{\hat b \hat d}\right), \ell\left(\cl{\hat b \hat c}\right) \le \beta$.  Arguing as in Case I, $\ell\left(\cl{\hat a \hat d}\right) =\alpha$ and  $\ell\left(\cl{\hat a \hat c}\right) \le \pi - \alpha$.
Thus the nontrivial loop  $\hat a\hat c \hat b \hat d$
 has $\ell(\hat a\hat c \hat b \hat d) \le  2\beta +\alpha +\pi -\alpha = \pi +2\beta < 2\pi$ since $\beta< \frac \pi 2$. and we have the same contradiction as before.
\end{proof}

\begin{Thm} \label{T:join} The join of two cantor sets is Tits rigid.
\end{Thm}
\begin{proof}

By Theorem \ref{T:pi/2}, $\alpha =\frac \pi 2$ and so by Corollary \ref{C:bound}, for any $a \in C_1$ and $b \in C_2$, $\ell\left(\overline {ab}\right) =\frac \pi 2$.

   Let $\hat Z= C_1*_S C_2 $ be the spherical join where the metric on $C_1$ is always $\pi$ for distinct points and similarly for $C_2$, so both are discrete as metric spaces.   Let $Z=C_1* C_2$ be the topological join (Notice that $C_1$ and $C_2$ are not discrete here).  Notice that as point sets, $\hat Z = Z$.

Define $\Phi: \hat Z \to \bd_T X$ by $\Phi$ being the identity on $C_1$ and $C_2$ and  $\Phi(c_1,c_2,t) = x$ where $x \in \cl{c_1c_2}$ with $d_T(c_1,x) =t$.  Note that $\Phi $ is an isometry.  We must show that $\Phi: Z \to \bd X$ is a homeomorphism (same point sets, different topologies).
Since $Z$ is compact and $\Phi$ is a bijection, it suffices to show that $\Phi$ is continuous.  Let $(z_k) \subset Z$ and with $z_k \to z$.
Pulling back to the product, we have $z=(a,b,t)$ and $z_k= (a_k,b_k,t_k)$  where $a,a_k \in C_1$, $b,b_k \in C_2$ and $t,t_k \in [0,\frac \pi 2]$.

We will show that $\Phi(z_k) \to \Phi(z)$.  For $t=0$, $\Phi(z) =a$.  Consider the sequence $(a_k) \subset C_1 \subset \bd X$. Since $a_k \to a$ and $d_T(a_k, \Phi(z_k)) = t_k$ by definition,  then $\Phi(z_k) \to a =\Phi(z) $ by lower semi-continuity of the Tits metric.  Similarly if $t = \frac \pi 2$.
When $t \in (0,\frac \pi 2)$ then $a_k \to a$ and $b_k \to b$ (not true in the other two cases).  By Theorem \ref{T:omega}, any cluster point $p$ of $(\Phi(z_k))$ lies on the suspension arc $\cl{ab}$.  By lower semi-continuity, $d_T(p, C_1) \le t$ and $d_T(p,C_2) \le \frac \pi 2 -t$.  It follows that
$d_T(p,C_1) = t$ and so $p = \Phi(z)$ so $\Phi(z_k) \to \Phi(z)$.
Thus $\Phi$ is a homeomorphism and a Tits isometry.

For any two such CAT(0) group boundaries, we get the boundary isomorphism by composing the "$\Phi$" from one with the "$\Phi^{-1}$" of the other.

\end{proof}
\section{Suspension of the join of two Cantor sets}

We have proven that the join of two Cantor sets $C_1$ and $C_2$ is Tits rigid. We want to prove that the suspension of it, i.e. $\sum (C_1 * C_2)$, is also Tits rigid.  We first need a result from dimension theory. We will use inductive dimension, which is equivalent to covering dimension in our setting.

 We define dim$\,\emptyset =-1$.
For  a point $z \in Z$, $Z$  has dimension $\le k$ at $p$ if for any neighborhood $U$ of $z$ there is a neighborhood $V \subset U$ of $z$ with dim$\, \bd V\le k-1$.  $Z$ has dimension $\le k$ if $Z$ has dimension $\le k$ at each point.

\begin{Lem}\label{L:dim} If $Z$ is a compact metrizable space of dimension $k$, then the suspension of $Z$, $\sum Z$ has dimension $k+1$.
\end{Lem}
\begin{proof} Since $Z$ is compact and $(0,1)$ is one dimensional then by \cite[page 34]{HU-WA}, dim$\,\left[Z \times (0,1)\right] =\,$dim$\,Z +\,$dim$\,(0,1)= k+1$.
So  $Z \times (0,1)$ has dimension $\le k+1$  at each point with equality at at least one point.  Thus $\sum Z$ has dimension $\le k+1$ at each point with  possible exceptions the suspension points $p$ and $n$.

Every neighborhood $U$ of $p$ will contain a cone neighborhood $V$ of $p$  with $\bd V \cong Z$.  Thus the dimension of $\sum Z$ at $p$ is at most  dim$\, Z +1=k+1$ and similarly for
$n$.  Since $\sum Z$ has dimension $\le k+1$ at each point with equality at at least one point, dim$\, \sum Z = k+1$.
\end{proof}

We now prove a result on the fixed point set of the group action on the boundary $\bd X$.
\begin{Lem}\label{L:helper} Let $G$ be a group acting geometrically on a CAT(0) space $X$.  If $G$ has a global fixed point $p$, then there is a closed convex quasi-dense
$\hat X \subset X$ with $\hat X = \R \times Y$ where $Y$ is a closed convex subset of $\hat X$ and $\R$ is an axis of a central element of $G$.
\end{Lem}
\begin{proof} The group $G$ is finitely generated by $g_1, \dots g_k$.  By \cite{RUA}, for each $i$, $p\in \fp(g_i) = \Lambda$Min$\,(g_i)= \Lambda Z_{g_i }$.  By \cite{SWE}
$$p \in \cap \Lambda Z_{g_i} = \Lambda\left[\cap Z_{g_i}\right] =\Lambda Z_G$$ and since $Z_G$ is convex by \cite{SWE}, $Z_G$ contains an element of infinite order $g$.  By \cite{RUA}, $g$ acts trivially on $\bd X$, so $\bd X =\fp(g) =\Lambda$Min$\,(g)$.  We now let $\hat X = $Min$\,(g)$ and apply \cite[II 6.8]{BRI-HAE}.
\end{proof}

\begin{Prop}\label{P:Rua}
Let $X$ be a  CAT(0) space, and $G$ be a group acting geometrically on $X$. The set $A$ of points virtually fixed by $G$ on the boundary is a Tits sphere, and $\bd X = A  * Z$ and $\bd_T X=A  *_S Z$ where $Z$ is a compact subset of $\bd X$.
\end{Prop}

\begin{proof}
If $A$ is non-empty, then passing to a subgroup of finite index, we may assume that the set of global fixed points of $G$ is non-empty.   By  Lemma \ref{L:helper}  there exists a hyperbolic element $h \in Z_G$ with endpoints $\{n,p\} \subset A$, and $X$ contains a quasi-dense subspace which is a product of an axis of $h$ with $X_1$,  where $X_1$ is a
 closed convex subspace on which $G/\langle h \rangle$ (see \cite{SWE}) acts geometrically.  Thus  $\bd X\cong\{n,p\}  * \bd X_1= \sum \bd X_1$, and $\bd_T X=\{n,p\} *_S\bd_TX_1= \sum_S \bd_TX_1$.

Suppose dim$\,\bd X=k$ ( $<\iy$ by \cite{SWE}). We  proceed by induction on $k$. For $k=0$, either $A =\varnothing$, or  $\bd X$ is a 0-sphere, because the 0-sphere is the only 0-dimensional space which is a suspension (of the empty set). In the latter case, $A=\bd X$ is a 0-sphere.

 Assume the result holds for dimension $k-1$. Let $\bd X$ be $k$-dimensional. If $A$ is empty, there is nothing to prove; if not, then $X$ contains a quasi-dense subspace $\mathbb R \times X_1$, with $\bd X=\{n,p\} *\bd X_1$ and $\bd_T X=\{n,p\}   *_S \bd_T X_1$.

 Since $\bd X_1 \subset \bd X$,  $\bd X_1$ is finite dimensional and by Lemma \ref{L:dim}, dim$\,\bd X_1=k-1$. Applying the result to $X_1$ with geometric action by $G/\langle h \rangle$, then the set $A_1$ of all points virtually fixed by $G/\langle h \rangle$ on $\bd X_1$ is a Tits sphere.  Also $\bd X_1=A_1* Z_1$ and $\bd_T X_1=A_1*_S Z_1$where $Z_1$ is a compact subset of $\bd X_1$. Any point virtually fixed by $G$ in $A-\{n,p\}$  lies on a suspension arc  through a unique point $q$ in $\bd X_1$. Thus $q$ is virtually fixed by $G$, and also by $G/\langle h \rangle$, so $q\in A_1$. It follows that $A$ is the spherical join of  $\{n,p\}$ with $A_1$, and so $A$ is a Tits sphere in $\bd_T X$ and $$\bd X = \{n,p\}*(A_1 *Z_1)= [\{n,p\} *A_1]*Z_1=  A*Z_1$$  with the same equalities for the spherical joins.
\end{proof}

\begin{Cor}\label{C:decomp}
Let $X$ be a  CAT(0) space, and $G$ be a group acting geometrically on $X$. Suppose that $\{n,p\}$ are points on $\bd X$ such that all homeomorphisms of $\bd X$  stabilize $\{n,p\}$, then the points $n$ and $p$ are the only virtually fixed points of $G$, and
there is a closed convex $Y\subset X$  and $R$ a geodesic line in $X$ satisfying:
\begin{itemize}
\item $R$ is a line from $n$ to $p$;
\item There is a closed convex quasi-dense subset $\hat X \subset X$ with $\hat X$ decomposing as $Y \times R$;
\item The CAT(0) space $Y$ admits a geometric action.
\end{itemize}

\end{Cor}
\begin{proof}
$G$ virtually fixes the points $n$ and $p$, so $n,p \in A$, the sphere of points virtually fixed by $G$
and $\bd X = A *Z$ for some closed subset of $Z$ of $\bd X$.   If $A \neq \{n,p\}$, then  since any
homeomorphism of $A$ with the identity map on $Z$ induces an homeomorphism on their join, there would be homeomorphisms of $\bd X$ that do not stabilize $\{n,p\}$, which contradicts the assumption. So
$A = \{n,p\}$ as required.  By Lemma \ref{L:helper} we get a closed convex quasi-dense $\hat X \subset X$ where $\hat X = Y \times R$ where $R$ is the axis of a central $g \in Z_G$.  Notice since $n$ and $p$ are the only virtually fixed points of $G$ and the endpoints of $R$ will be virtually fixed by $G$, then $n$ and $p$ are the endpoints of $R$.  Also $G/\langle g\rangle$ will act geometrically on $Y$ by \cite{SWE}.
\end{proof}
We need the following characterization of an arc.
\begin{Thm}[Moore] \label{T:Moore} Let $A$ be a compact connected metric space.  If $A$ has exactly two non cut points, then $A$ is an arc.\end{Thm}

\begin{Lem}\label{L:isotop} Let $Y$ be the join of two cantor sets $C_1$ and $C_2$. Then the suspension point set $\{n,p\}$ is preserved by homeomorphisms of $\sum Y$.
If $\sum Y$ is also a suspension of a subspace $Z$, then $n$ and $p$ are the $Z$-suspension points as well and  $Z$ is isotopic to $Y$ in $\sum Y$.
\end{Lem}
\begin{proof}  The suspension arcs of $\sum Y$ will be called $Y$-suspension arcs.  Similarly we will call the suspension arc of the $Z$ suspension structure $Z$-suspension arcs.
We partition $\sum Y$ by the local topology.
\begin{itemize}
\item The suspension points $\{n,p\}$ which have a neighborhood basis consisting of cone neighborhoods (cones on $Y$ of course), so $\sum Y$ is locally connected at $n$ and $p$.
\item$\cC=[\sum C_1\cup \sum C_2]-\{n,p\}$ (the union of the open $Y$-suspension arcs running through $C_1$ and $C_2$). $\sum Y$ is not locally connected at these points. For $p\in \cC$ and  $U$ a neighborhood of $p$,  the component of $U$ containing $p$ is never a 2 manifold (it always contains a topological tri-plane)
\item $\cD =\sum Y -[\sum C_1\cup \sum C_2]$. $\sum Y$ is not locally connected at these points.   For $p\in \cD$, for $U$ a sufficiently small neighborhood of $p$, the component of $p$ in $U$ will be homeomorphic to an open subset of a disk.
\end{itemize}
  This means that the suspension points of the $Z$-suspension are also $\{n,p\}$ and that this set is fixed by every homeomorphsim of $Y$.

Since we can isotop up and down suspension arcs, if $\alpha$ is an open $Z$-suspension arc and $\alpha  \cap \cC \neq \emptyset$, then $\alpha \subset \cC$. (Similarly if $\alpha \cap \cD \neq \emptyset$ then $\alpha \subset \cD$.) It follows that for each $c \in C_1\cup C_2$, the $Y$-suspension arc through $c$ will be a $Z$-suspension arc as well.  Let $c_1 \in C_1$ and $c_2\in C_2$.  Let $\beta \subset C_1*C_2$ be the join arc from $c_1$ to $c_2$.
The disk $D=\sum \beta \subset \sum (C_1*C_2)=Y$ has boundary $\alpha_1$ and $\alpha_2$, the ($Y$ and $Z$)-suspension arcs throughout $c_1$ and $c_2$ respectively.  Let $\omega = Z \cap D$.   Since $\alpha_i$ is a $Z$-suspension arc, $\alpha_i \cap \omega$ is a single point $z_i$ (for $i=1,2$).
We will show that $\omega$ is an arc by showing that $\omega$ is connected and that $z_1,z_2$ are the only non cut points of $\omega$.

Open $Z$-suspension arcs are disjoint, so $D = \sum \omega$.  Since $D-\{n,p\}\cong \omega \times I$ (where $I$ is a open interval) is connected, $\omega$ is connected.  Notice that
$\sum [\omega- \{z_i\}]-\{n,p\} \cong [D-\alpha_i]$ which is connected and so $\omega -\{z_i\}$ is connected.
Let $z\in \omega$ with $z \neq z_i$ for $i=1,2$.  Thus $z \in\, $Int$\, D$ and so the $Z$-suspension arc $\gamma\subset D$  and
 $\gamma \cap \bd D = \{n,p\}$ since $\gamma$ cannot cross $\alpha_i$.  It follows that $D- \gamma$ is not  connected.
 Since $D-\gamma \cong [\omega-\{z\}] \times I$  it follows that $\omega -\{z\}$ is not connected.  Thus $z$ is a cut point of $\omega$,
 and by Theorem \ref{T:Moore}, $\omega$ is an arc.

 Notice that $D$ admits a PL structure as a square with vertices  $n,c_1,p,c_2$ and so that the map from $\cl {c_i n}$ and $\cl{c_i,p} $ into $D$ is an isometry, and we can do this in a canonical way for all such $D$.
 We isotope $\omega$ to  the line segment $\hat \omega$ in $D$ from $z_1$ to $z_2$ with the isotopy fixing $\bd D$.  We can do this for each such $D$ at the same time.  We call the image $Z$ under this isotopy $\hat Z$ which is a union of straight line segments in each of our squares.

 Now for each $c \in C_1\cup C_2$, with $z$ be the unique point of $Z$ on the $Y$-suspension arc $\alpha$ through $c$, we choose an isotopy of $\alpha$ fixing $n$ and $p$ which takes $z$ to $c$.  We do these simultaneously and extend linearly on corresponding squares.  This gives us an isotopy from $\hat Z$ to $Y$ in $\sum Y$.
\end{proof}

\begin{Thm} The suspension of the join of two cantor sets is Tits rigid.
\end{Thm}
\begin{proof}Let $G$ be a group acting geometrically on the CAT(0) space $X$ with $\bd X \cong \sum [C_1 * C_2]$ where $C_1$ and $C_2$ are Cantor sets.  We will show that there is a isometry from $\iota:\bd_TX\to\sum_S [C_1 *_S C_2]$  such that $\iota$ is a homeomorphsim from $\bd X$ to  $ \sum [C_1 * C_2]$.

By Lemma \ref{L:isotop} every homeomorphism of $\bd X$ fixes the suspension point set $\{n,p\}$.
Thus by
Corollary \ref{C:decomp}, there exits closed convex quasi-dense $\hat X \subset X$ with $\hat X = Y \times R$ where $Y$ is closed and convex in $X$ and
$R$ is a geodesic line from $n$ to $p$. Also $Y$ admits a geometric action.  Now $\bd X = \sum \Lambda Y$ and $\bd_T X = \sum_S \Lambda Y$.  By  Lemma  \ref{L:isotop}, $\Lambda Y$ is the join of two cantor sets.  By Lemma \ref{L:iota}, $\bd Y$ is is the join of two cantor sets.  By Theorem \ref{T:join},  $\bd_TY$ is the spherical join of two cantor sets $B_1$ and $B_2$ where $\bd Y$ is the topological join of $B_1$ and $B_2$.
By Lemma \ref{L:iota},  in the restriction Tits metric, $\Lambda Y\cong B_1*_SB_2$.
Thus $\bd_T X $ is isometric to  $\sum_S B_1*_SB_2$ and this gives $\iota$ as required.

\end{proof}

\section{Further questions}
We would recklessly conjecture that a boundary is Tits rigid if and only if it doesn't have a circle as a join factor.  Clearly the circle is not Tits rigid and all  higher dimensional spheres will have a circle as a join factor.  Thus every known non-Tits rigid space has a circle as a join factor.
Furthermore if we have a boundary $Z\cong S^1 *Y$  where $Y$ {\em is also a boundary}, we see taking products that  $Z$ will not be Tits rigid.

The first step is to prove more boundaries that are Tits rigid. Possible candidates include the $n$-fold join of Cantor sets and their suspensions, or more generally, boundaries of CAT(0) cube complexes with certain properties. Visual boundaries of universal covers of Salvetti complexes of right-angled Artin groups may be a source of examples, because the known examples can be realized as such.  We also imagine that spherical buildings are Tits rigid, but have not examined this at all.

The known Tits rigid boundaries have proper closed invariant subsets, except the Cantor set and the set with two points. This may be a common property for other Tits rigid boundaries. For those that do have closed invariant subsets, is it true that a Tits rigid boundary always has some closed invariant subset which is also Tits rigid with the induced topology?

Also, for a Tits rigid boundary that is not a suspension, is the suspension of this boundary is also Tits rigid? Corollary \ref{C:decomp} is a partial result. The difficulty lies in the fact that there may be non-homeomorphic topological spaces with homeomorphic joins, an example is given by the Double Suspension Theorem of Cannon and Edwards.

In every known Tits rigid boundary, the topology of the Tits boundary resembles the visual topology in the best possible way, i.e. all the paths in the visual topology are still paths in the Tits boundary. We suspect that this may be the case for other Tits rigid boundary as well, although this should become much harder to show even for particular cases when the dimension of the visual boundary is at least two.

Recall that a CAT(0) group is rigid if it corresponds to a unique visual boundary up to homeomorphism.  There might be non-rigid CAT(0) groups corresponding to some Tits rigid boundaries. Boundaries of known  non-rigid CAT(0) groups are similar, in the sense that they are shape equivalent by a result of Bestvina. If one of them is Tits rigid, should every other visual boundary of the same group be Tits rigid too because of their similarity?

\end{document}
\bye

%% file: proofcase1.pdf_tex

\begingroup
  \makeatletter
  \providecommand\color[2][]{%
    \errmessage{(Inkscape) Color is used for the text in Inkscape, but the package 'color.sty' is not loaded}
    \renewcommand\color[2][]{}%
  }
  \providecommand\transparent[1]{%
    \errmessage{(Inkscape) Transparency is used (non-zero) for the text in Inkscape, but the package 'transparent.sty' is not loaded}
    \renewcommand\transparent[1]{}%
  }
  \providecommand\rotatebox[2]{#2}
  \ifx\svgwidth\undefined
    \setlength{\unitlength}{466.04942017pt}
  \else
    \setlength{\unitlength}{\svgwidth}
  \fi
  \global\let\svgwidth\undefined
  \makeatother
  \begin{picture}(1,0.36012267)%
    \put(0,0){\includegraphics[width=\unitlength]{proofcase1.pdf}}%
    \put(0.11388737,0.3387027){\color[rgb]{0,0,0}\makebox(0,0)[lb]{\smash{$b$}}}%
    \put(0.33320704,0.33946422){\color[rgb]{0,0,0}\makebox(0,0)[lb]{\smash{$a$}}}%
    \put(0.08571087,0.00363092){\color[rgb]{0,0,0}\makebox(0,0)[lb]{\smash{$c$}}}%
    \put(0.23497011,0.00439244){\color[rgb]{0,0,0}\makebox(0,0)[lb]{\smash{$d$}}}%
    \put(0.41697501,0.00591549){\color[rgb]{0,0,0}\makebox(0,0)[lb]{\smash{$e$}}}%
    \put(0.0932865,0.18952377){\color[rgb]{0,0,0}\makebox(0,0)[lb]{\smash{$m$}}}%
    \put(0.7365573,0.34427135){\color[rgb]{0,0,0}\makebox(0,0)[lb]{\smash{$\hat m$}}}%
    \put(0.61546565,0.31007276){\color[rgb]{0,0,0}\makebox(0,0)[lb]{\smash{$\hat b$}}}%
    \put(0.87055685,0.29078424){\color[rgb]{0,0,0}\makebox(0,0)[lb]{\smash{$\hat c$}}}%
    \put(0.92408262,0.14467329){\color[rgb]{0,0,0}\makebox(0,0)[lb]{\smash{$\hat a$}}}%
    \put(0.62559211,0.11043609){\color[rgb]{0,0,0}\makebox(0,0)[lb]{\smash{$\hat d = \hat e$}}}%
    \put(0.38826455,0.18013143){\color[rgb]{0,0,0}\makebox(0,0)[lb]{\smash{$\alpha$}}}%
    \put(0.73545873,0.27946758){\color[rgb]{0,0,0}\makebox(0,0)[lb]{\smash{$\alpha$}}}%
    \put(0.61201195,0.19749113){\color[rgb]{0,0,0}\makebox(0,0)[lb]{\smash{$\alpha$}}}%
    \put(0.16241295,0.20109234){\color[rgb]{0,0,0}\makebox(0,0)[lb]{\smash{$\alpha$}}}%
    \put(-0.00017859,0.26809885){\color[rgb]{0,0,0}\makebox(0,0)[lb]{\smash{$C_1$}}}%
    \put(-0.00106614,0.06650827){\color[rgb]{0,0,0}\makebox(0,0)[lb]{\smash{$C_2$}}}%
  \end{picture}%
\endgroup

%% file: proofcase2.pdf_tex

\begingroup
  \makeatletter
  \providecommand\color[2][]{%
    \errmessage{(Inkscape) Color is used for the text in Inkscape, but the package 'color.sty' is not loaded}
    \renewcommand\color[2][]{}%
  }
  \providecommand\transparent[1]{%
    \errmessage{(Inkscape) Transparency is used (non-zero) for the text in Inkscape, but the package 'transparent.sty' is not loaded}
    \renewcommand\transparent[1]{}%
  }
  \providecommand\rotatebox[2]{#2}
  \ifx\svgwidth\undefined
    \setlength{\unitlength}{463.78683801pt}
  \else
    \setlength{\unitlength}{\svgwidth}
  \fi
  \global\let\svgwidth\undefined
  \makeatother
  \begin{picture}(1,0.36187953)%
    \put(0,0){\includegraphics[width=\unitlength]{proofcase2.pdf}}%
    \put(0.10956447,0.34035506){\color[rgb]{0,0,0}\makebox(0,0)[lb]{\smash{$b=g^+$}}}%
    \put(0.32995409,0.3411203){\color[rgb]{0,0,0}\makebox(0,0)[lb]{\smash{$a$}}}%
    \put(0.08125052,0.00364863){\color[rgb]{0,0,0}\makebox(0,0)[lb]{\smash{$c$}}}%
    \put(0.23123792,0.00441387){\color[rgb]{0,0,0}\makebox(0,0)[lb]{\smash{$d$}}}%
    \put(0.41413073,0.00594435){\color[rgb]{0,0,0}\makebox(0,0)[lb]{\smash{$e$}}}%
    \put(0.0888631,0.19044836){\color[rgb]{0,0,0}\makebox(0,0)[lb]{\smash{$m$}}}%
    \put(0.73527209,0.34595087){\color[rgb]{0,0,0}\makebox(0,0)[lb]{\smash{$\hat m$}}}%
    \put(0.6135897,0.31158545){\color[rgb]{0,0,0}\makebox(0,0)[lb]{\smash{$\hat b$}}}%
    \put(0.86992536,0.29220282){\color[rgb]{0,0,0}\makebox(0,0)[lb]{\smash{$\hat c$}}}%
    \put(0.92371226,0.14537908){\color[rgb]{0,0,0}\makebox(0,0)[lb]{\smash{$\hat a$}}}%
    \put(0.62376557,0.11097485){\color[rgb]{0,0,0}\makebox(0,0)[lb]{\smash{$\hat d = \hat e$}}}%
    \put(0.03099452,0.14010501){\color[rgb]{0,0,0}\makebox(0,0)[lb]{\smash{$\beta$}}}%
    \put(0.73416816,0.28083096){\color[rgb]{0,0,0}\makebox(0,0)[lb]{\smash{$\alpha$}}}%
    \put(0.61011915,0.19845459){\color[rgb]{0,0,0}\makebox(0,0)[lb]{\smash{$\alpha$}}}%
    \put(0.15832679,0.20207337){\color[rgb]{0,0,0}\makebox(0,0)[lb]{\smash{$\beta$}}}%
    \put(0.00088536,0.27010194){\color[rgb]{0,0,0}\makebox(0,0)[lb]{\smash{$C_1$}}}%
    \put(0.00088536,0.07119791){\color[rgb]{0,0,0}\makebox(0,0)[lb]{\smash{$C_2$}}}%
  \end{picture}%
\endgroup

%% file: Tits_rigid_201403.bbl
\begin{thebibliography}{99}
\bibitem{BAL}
W. Ballmann {\em Lectures on Spaces of Nonpositive Curvature}
DMV Semminar, Band 25, Birkhaeuser, 1995.







\bibitem{BRI-HAE}
M. Bridson and A. Haefliger {\em Metric spaces of non-positive
curvature} Grundlehren der Mathematischen Wissenschaften
[Fundamental Principles of Mathematical Sciences], 319.
Springer-Verlag, Berlin, 1999. xxii+643 pp. ISBN: 3-540-64324-9

\bibitem{CHA}
K. Chao {\em CAT(0) spaces with boundary the join of two Cantor sets} Accepted by Algebraic and Geometric Topology.


\bibitem{GU-SW}
D. Guralnik and E. Swenson {\em A `transversal' for minimal invariant sets in the boundary of a CAT(0) group}
Trans. Amer. Math. Soc. 365 (2013), 3069--3095.
\bibitem{HU-WA} W. Hurewicz and H. Wallman {\em Dimension Theory}
Princeton University Press 1948.


\bibitem{RUA}
K. Ruane, {\em Dynamics of the action of a ${\rm CAT}(0)$ group on
the boundary}, Geom. Dedicata 84 (2001), no. 1-3, 81--99.
%
\bibitem{RUA1}
Kim~E. Ruane.
\newblock {CAT}(0) groups with specified boundary.
\newblock {\em Algebraic and Geometric Topology}, 6:633--649, 2006.

%
%
\bibitem{SWE}
E. Swenson, {\em A cut point theorem for ${\rm CAT}(0)$ groups},
J. Differential Geom. 53 (1999), no. 2, 327--358
\end{thebibliography}
